\documentclass{amsart}

\usepackage{amsmath,amsfonts,amssymb,amsthm}
\usepackage{graphicx}
\usepackage[colorlinks=true, allcolors=blue]{hyperref}
\usepackage{tikz}
\usetikzlibrary{calc}
\pgfdeclarelayer{background}
\pgfdeclarelayer{foreground}
\pgfsetlayers{background,foreground}

\newtheorem{thm}{Theorem}
\newtheorem{lem}[thm]{Lemma}
\newtheorem{defn}[thm]{Definition}

\newtheorem{prop}[thm]{Proposition}

\newcommand{\erf}{{\rm erf}}
\newcommand{\sat}{{\rm sat}}
\newcommand{\sm}{{\rm small}}
\newcommand{\lr}{{\rm large}}
\newcommand{\F}{{\mathcal F}}
\newcommand{\A}{{\mathcal A}}

\title{Saturation of $k$-chains in the Boolean lattice}
\author{Ryan R. Martin}
\address{Department of Mathematics, Iowa State University, Ames, Iowa, USA}
\author{Nick Veldt}
\email{rymartin@iastate.edu, mnveldt@iastate.edu}

\begin{document}

\maketitle

\begin{abstract}
Given a set $X$, a collection $\mathcal{F} \subset \mathcal{P}(X)$ is said to be $k$-Sperner if it does not contain a chain of length $k+1$ under set inclusion and it is $saturated$ if it is maximal with respect to this property. Gerbner et al. conjectured that, if $|X|$ is sufficiently large compared to $k$, then the minimum size of a saturated $k$-Sperner system is $2^{k-1}$. Noel, Morrison, and Scott disproved this conjecture later by proving that there exists $\varepsilon$ such that for every $k$ and $|X| > n_0(k)$, there exists a saturated $k$-Sperner system of cardinality at most $2^{(1-\varepsilon)k}$.

In particular, Noel, Morrison, and Scott proved this for $\varepsilon = 1-\frac{1}{4}\log_2 15 \approx 0.023277$. We find an improvement to $\varepsilon = 1-\frac{1}{5}\log_2 28 \approx 0.038529$. We also prove that, for $k$ sufficiently large, the minimum size of a saturated $k$-Sperner family is at least $\sqrt{k}\, 2^{k/2}$, improving on the previous Gerbner, et al. bound of $2^{k/2-0.5}$
\end{abstract}

\section{Introduction}

Given a set $X$, we define a \emph{Sperner system} as a collection $\mathcal{F} \subseteq \mathcal{P}(X)$ such that there does not exist $A, B \in \mathcal{F}$ such that $A \subsetneq B$. A $k$-\emph{chain} (or a \emph{chain}, where $k$ is understood) is a subcollection $\{A_1, \ldots, A_k\}$ such that $A_1 \subsetneq \cdots \subsetneq A_k$.

More generally, a $k$-\emph{Sperner system} is a collection $\mathcal{F} \subseteq \mathcal{P}(X)$ with no $(k+1)$-chain. The size of the largest Sperner system (i.e., $1$-Sperner system or antichain) was determined by Sperner~\cite{Spe} and the size of the largest $k$-Sperner system was determined by Erd\H{o}s~\cite{Erd}. Extremal poset theory is an increasingly well-researched subject in which one asks for the largest family in an $n$-dimensional Boolean lattice that has no copy of some given poset $P$. For a survey of some relatively recent results on this subject, see~\cite{GriggsLi} and~\cite[Chapter 7]{GPbook}. 

The above consists of so-called Tur\'an~\cite{Turan,ErdSto,ErdSim} problems in the Boolean lattice and so it is natural to consider saturation problems as well, in the style of Erd\H{o}s, Hajnal, and Moon~\cite{EHM}. 

We say that a $k$-Sperner system is \emph{saturated} if, for every set $S \in \mathcal{P}(X) \setminus \mathcal{F}$, the collection $\mathcal{F} \cup \{S\}$ contains a $(k+1)$-chain. 
%A saturated $1$-Sperner system is also referred to as a \emph{saturated antichain}, that is, a collection of unrelated items such that any item not in the saturated antichain will relate to at least one of them.

For a set $X$ of cardinality $n$, we are interested in finding the size of the smallest saturated $k$-Sperner system. This problem was first studied by Gerbner, Keszegh, Lemons, Palmer, P\'alv\"olgyi, and Patk\'os~\cite{GERB}, and later improved by Morrison, Noel, and Scott~\cite{MNS}. Saturation was studied for other posets, in particular a version known as \emph{induced saturation}~\cite{FKKMRSS,KLMPP,MSW,FPSST,BGJJ,BGIJ,IvanButterfly,IvanDiamond,Liu}. It should be noted that, in the case of chains, induced saturation is the same as ordinary (also called \emph{weak}) saturation. 

Given integers $n$ and $k$, let $\sat(n,k)$ denote the minimum size of a saturated $k$-Sperner system in $\mathcal{P}(X)$ where $|X| = n$. It was shown by Gerbner, et al.~\cite{GERB}, that for sufficiently large $n$, the saturation number will not change as the cardinality increases. So, we define 
$$
\sat(k) := \lim_{n\rightarrow \infty} \sat(n,k) .
$$

For the rest of this paper, we study $\sat(k)$.

\begin{thm}[Gerbner et al.~\cite{GERB}]
For $k\in\{1,2,3\}$, $\sat(k) = 2^{k-1}$, and $\sat(k) \leq 2^{k-1}$ for all $k\geq 4$.
\end{thm}
The upper bound is accomplished via construction: Fix $k$, and consider the power set of $[k-2]$, which we denote $2^{[k-2]}$. Let $\F$ consist of the sets $A$ and $[n]-A$ for every $A\in 2^{[k-2]}$. The family $\F$, of $2^{k-1}$ members of $2^{[n]}$, is a saturated $k$-Sperner system, giving the upper bound. 

For example, for $k=4$, $\F=\left\{\emptyset, \{1\}, \{2\}, \{1,2\}, \linebreak\overline{\{1,2\}}, \overline{\{1\}}, \overline{\{2\}}, \overline{\emptyset}\right\}$, which is a saturated 4-Sperner system. 

\begin{thm}[Morrison, Noel, Scott~\cite{MNS}]
If $k \geq 6$, then $\sat(k) \leq 2^{(1-\varepsilon)k}$, where $\varepsilon = 1-\frac{1}{4}\log_2 15 \approx 0.023277$. If $k\leq 5$, then $\sat(k) = 2^{k-1}$.
\end{thm}

This improvement on the upper bound is due to a more efficient saturated set for $k=6$ which appeared in~\cite{GERB}, and they showed that it that can be extended for larger values of $k$. 

In this paper, we provide a tightening of the bounds as follows:

\begin{thm}
If $k \geq 7$, then $\sat(k) \leq 2^{(1-\varepsilon)k}$, where $\varepsilon = 1-\frac{1}{5}\log_2 28 \approx 0.038529$.
\label{thm:UB}
\end{thm}

Similar to~\cite{MNS}, we obtain a more efficient saturated set for $k=7$, which by the same method can be extended for larger values of $k$.

As to the lower bounds for $\sat(k)$, these also first appeared in~\cite{GERB}.

\begin{thm}[Gerbner et al.~\cite{GERB}]
    For $k\geq 1$, $\sat(k)\geq 2^{k/2-0.5}$.
\end{thm}

This lower bound comes from a counting argument. Let $\mathcal{F}$ be a saturated $k$-Sperner set and consider some element $S \in 2^{[n]}\setminus \mathcal{F}$. By definition, there exists a sequence of sets $F_1, \ldots, F_{k}$ such that $F_1 \subset \cdots \subset F_i \subset S \subset F_{i+1} \subset \cdots \subset F_k$ is a $k+1$-chain. Because of the existence of the $k$-chain in $\mathcal{F}$, the size difference $|F_{i+1} \setminus F_i|$ is at most $n-k+2$, meaning that each pair of elements in $\mathcal{F}$ can account for at most $2^{n-k+2}$ elements of $2^{[n]}$, including $F_i$ and $F_{i+1}$ themselves. Thus, 
$$
{|\F| \choose 2} 2^{n-k+2} \geq 2^n .
$$ 
Rearranging gives the lower bound.

In this paper, we use probabilistic methods to slightly improve this bound for large $k$. 
\begin{thm}
    If $k\geq 497$, then $\sat(k) > \sqrt{k}\, 2^{k/2}$.
    \label{thm:LB}
\end{thm}

\section{Preliminaries}

Given a collection $\mathcal{F} \subseteq \mathcal{P}(X)$, we say that a set $A \subseteq X$ is an \emph{atom} for $\mathcal{F}$ if $A$ is maximal with respect to the property that for every set $S \in \mathcal{F}$, $S\cap A \in \{\emptyset,A\}$. In other words, each set in our collection must contain all of $A$, or none of it.

We say that an atom $A$ is \emph{homogeneous} if $|A| \geq 2$, and we say that an element $S \in \mathcal{F}$ is a \emph{singleton} if $|S|=1$.

\begin{lem}[Gerbner et al.~\cite{GERB}]
    Given a set $X$ and a collection $\mathcal{F} \subseteq \mathcal{P}(X)$, if $|X| > 2^{|\mathcal{F}|}$, then there is at least one homogeneous atom in $\mathcal{F}$. Furthermore, if $\mathcal{F}$ is a saturated $k$-Sperner system and $H_1$ and $H_2$ are homogeneous for $\mathcal{F}$, then $H_1 = H_2$.\label{lem:homogeneous}
\end{lem}
For our purposes, this means that if $n$ is large enough compared to $k$, there will be exactly one homogeneous atom in our saturated $k$-Sperner system. Since we are improving the upper bound from $2^{k-1}$, Lemma~\ref{lem:homogeneous} gives that it is more than sufficient to assume $n$ is of size $2^{2^{k-1}}$. We will distinguish elements in such a system by whether or not they contain the homogeneous atom.

    \begin{defn}
         Let $\mathcal{F}$ be a saturated $k$-Sperner system and let $H$ be homogeneous in $\mathcal{F}$. We say that a set $S \in \mathcal{F}$ is \emph{large} if $H \subseteq S$ or \emph{small} if $S \cap H = \emptyset$. We write the collection of such sets as $\mathcal{F}^\lr$ and $\mathcal{F}^\sm$, respectively. 
    \end{defn}

    Lemma~\ref{lem:satant} establishes a stronger condition on saturated antichains as long as there is a homogeneous set. 

\begin{lem}[Morrison, Noel, Scott~\cite{MNS}]
    Let $\A \subseteq \mathcal{P}(X)$ be a saturated antichain with homogeneous atom $H$. Then every set $S \in \mathcal{P}(X) \setminus \A$ either contains a set in $\A^\sm$ or is contained in a set of $\A^\lr$.\label{lem:satant}
\end{lem}

There is a natural way to partition a $k$-Sperner system $\mathcal{F}$ into a sequence of $k$ pairwise disjoint antichains by iteratively pulling off antichains of minimal elements. 

\begin{defn} For $0 \leq i \leq k-1$, let $\A_i$ be the collection of all minimal elements of $\mathcal{F}\setminus (\bigcup_{j<i}\A_j)$ under inclusion. We say that $(\A_i)_{i=0}^{k-1}$ is the \emph{canonical decomposition} of $\mathcal{F}$ into antichains. 
\end{defn}

\begin{defn}
    A collection $(\mathcal{D}_i)_{i=0}^t$ of subsets is \emph{layered} if, for $1 \leq i \leq t$, every $D \in \mathcal{D}_i$ strictly contains some $D' \in \mathcal{D}_{i-1}$ as a subset.
\end{defn}

Note that the canonical decomposition of any set system is automatically layered. Lemma~\ref{lem:antichainsareSperner} establishes that if one has a layered collection of disjoint saturated antichains, then it is a $k$-Sperner system.

\begin{lem}[Morrison, Noel, Scott~\cite{MNS}]
    If $(\A_i)_{i=0}^{k-1}$ is a layered collection of pairwise disjoint saturated antichains in $\mathcal{P}(X)$, then $\mathcal{F} := \bigcup_{i=0}^{k-1}\A_i$ is a saturated $k$-Sperner system. %Furthermore, every $A \in \A_i$ is strictly contained in some $B \in \A_{i+1}$. 
    \label{lem:antichainsareSperner}
\end{lem}

So, then, to find a good saturated $k$-Sperner system, it suffices to find a collection of layered pairwise disjoint saturated antichains. 

\begin{lem}[Morrison, Noel, Scott~\cite{MNS}]
    A collection of pairwise disjoint saturated antichains $(\A_i)_{i=0}^{k-1}$, which have the same homogeneous atom, is layered if and only if $(\A_i^\sm)_{i=0}^{k-1}$ is layered.
    \label{lem:smallimplieslayered}
\end{lem}

Notably, if we combine this condition with the layered collection condition and guarantee the family has a homogeneous set, then two conditions are sufficient to verify the upper bound construction in Section~\ref{sec:upperbound}:
\begin{enumerate}
    \item $\A_i$ must be a saturated antichain.
    \item Each small set in $\A_i$ must contain a small element in $\A_{i-1}$, $1\leq i\leq k-2$.
%    \item Each large set in $\A_i$ must be contained in a large element in $\A_{i+1}$, $i=1,\ldots,k-2$.
\end{enumerate}

By Lemma~\ref{lem:antichainsareSperner}, it sufficies to have a collection of layered antichains and by Lemma~\ref{lem:smallimplieslayered}, the layered condition is verified if each element in $\A_i^\sm$ strictly contains an element in $\A_{i-1}^\sm$, $1\leq i\leq k-2$. 

As to the lower bound, Lemma~\ref{lem:decomp} gives that any saturated $k$-Sperner system with a homogeneous set has a canonical decomposition into saturated antichains and so any such system with fewer than $2^{k-1}$ elements has such a decomposition, provided $n$ is large enough. 

\begin{lem}[Morrison, Noel, Scott~\cite{MNS}]
    Let $\mathcal{F} \in \mathcal{P}(X)$ be a saturated $k$-Sperner system with homogeneous atom $H$ and canonical decomposition $(\A_i)_{i=0}^{k-1}$. Then $\A_i$ is saturated for all $i$. \label{lem:decomp}
\end{lem}

\section{Further Characteristics of Layered Saturated Antichain}

First, we observe that the empty set and the complete set $[n]$ will always be contained in a minimal $k$-Sperner system for $k\geq 2$. Each forms a saturated antichain of size 1, which is as small as possible for $\A_0$ and $\A_{k-1}$, respectively.
\begin{prop}
    Let $\mathcal{F} \in \mathcal{P}(X)$ be a minimal saturated $k$-Sperner system with canonical decomposition $(\A_i)_{i=0}^{k-1}$. Then $\A_0=\{\emptyset\}$, $\A_{k-1}=\{[n]\}$, and for $i\in\{1,\ldots,k-2\}$, each small element $A \in \A_i^\sm$ will have size at least $i$, and each large element $A \in \A_i^\lr$ will have size at most $i+n-(k-1)$.
    \label{prop:sizes}
\end{prop}
\begin{proof}
    Because of the layered condition, every set in $\A_i^\sm$ must have at least one more element than some set in $\A_{i-1}^\sm$, and every set in $\A_i^\lr$ must have at least one fewer element than some set in $\A_{i+1}^\lr$. So the result follows by induction. 
\end{proof}
\begin{lem}
    Let $\mathcal{F} \in \mathcal{P}(X)$ be a minimal saturated $k$-Sperner system  with canonical decomposition $(\A_i)_{i=0}^{k-1}$ and homogeneous atom $H$. Then $\A_1$ will consist of at least $k-2$ singleton atoms in $\A_1^{\sm}$, and one element in $\A_1^\lr$.
    By symmetry, $\A_{k-2}$ will consist of at least $k-2$ elements of size $n-1$ in $\A_{k-2}^{\lr}$ and one element in $\A_{k-2}^{\sm}$.
    \label{lem:A1}
\end{lem}
\begin{proof}
    We first prove that $\A_1^{\sm}$ will contain only singletons. Suppose that $\mathcal{F}$ is a minimal saturated $k$-Sperner system, and $\A_1$ contains small elements of size $>1$. Consider the following reduction operation on antichains:
    \begin{quote}
        \textbf{Reduction operation:} 
        \begin{enumerate}
            \item   Set $\A^{\sm}=\A_1^{\sm}$ and $\A^{\lr}=\A_1^{\lr}$
            \item  Choose an atom, $x$, in some non-singleton element, $\{x\}\cup S\in\A^{\sm}$. \label{it:chooseatom}
            \item Replace $\{x\}\cup S$ with $\{x\}$ in $\A$.\label{it:recursepoint}
            \item If $x\in A$ for $A\in\A-\left\{\{x\}\right\}$, then replace $A$ with $A-\{x\}$ in $\A$.\label{it:eliminatex}
            \item Resolve any pairs of elements that contain each other, $A\subset B$:\label{it:cases}
                \begin{itemize}
                    \item If $A,B\in\A^{\sm}$, then remove $B$ from $\A$. 
                    \item If $A,B\in\A^{\lr}$, then remove $A$ from $\A$.
                    \item If $A\in\A^\sm$ and $B\in\A^\lr$, then reassign $x$ to be an arbitrary atom of $A$ and go to \eqref{it:recursepoint}.
                \end{itemize}
            \item If $\A^{\sm}$ consists of singletons, stop and return $\A$. Otherwise, go to~\eqref{it:chooseatom}.
        \end{enumerate}
    \end{quote}

    Note that if $\A_1^{\sm}$ consists only of atoms, then the reduction operation does nothing. 

    We must verify that:
        \begin{itemize}
            \item The reduction operation terminates.
            \item The resulting set $\A$ obeys $\bigl|\A\bigr|\leq \bigl|\A_{1}\bigr|$.
            \item The resulting set $\A$ is an antichain.
            \item The resulting set $\A$ is saturated
            \item The resulting set $\A^{\sm}$ consists of singletons.
            \item $\A_0$, $\A$, $(\A_i)_{i=2}^{k-1}$ is layered.
        \end{itemize}

    As to the verification,
    \begin{itemize}
        \item The algorithm terminates because step~\eqref{it:eliminatex} ensures that no other member of $\A$ will contain the atom $x$ so each iteration of the inner loop will use a different one of the finite set of atoms. The outer loop terminates because we reduce the number of non-singleton elements of $\A^{\sm}$ by at least 1 for each iteration. 
        \item The set $\A$ obeys $\bigl|\A\bigr|\leq \bigl|\A_{1}\bigr|$ because throughout the algorithm, no elements are added to $\A$ and some may be removed.
        \item The set $\A$ is an antichain because step~\eqref{it:cases} ensures that containments will be eliminated either directly or by step~\eqref{it:eliminatex}. 
        \item After each step of the algorithm, the set $\A$ is always saturated. To see this, consider any element $T \notin \A\setminus\left(\{x\}\cup S\right)$ at the start of any step.
        \begin{itemize}
            \item For step~\eqref{it:recursepoint}, if $T\supset \{x\}\cup S$, then $T\supset\{x\}$. If $T \subset A$ for some $A \in \A^\lr$, then this will continue to hold. If $T=\{x\} \cup S$, then $T$ is no longer in $\A$, but contains $\{x\}$.
            \item For step~\eqref{it:eliminatex}, if $T\supset A$, then $T\supset A-\{x\}$. If $T \subset A$, then either $T \subset A-\{x\}$ or $T\supset\{x\}$. If $T=A$, then $T$ is no longer in $\A$, but contains $\{x\}$.
            \item For step~\eqref{it:cases}, if $T\supset B\supset A$, then $T\supset A$. If $T \subset A \subset B$, then $T \subset B$. If $A,B\in\A^{\sm}$ and $T=B$, then $T\supset A$. If $A,B\in\A^{\lr}$ and $T=A$, then $T\subset B$.
        \end{itemize}
        So after each step, $\A$ still saturates all elements not in $\A$.
        \item The set $\A^{\sm}$ consists of singletons because if $\A^{\sm}$ ever contains a non-singleton, then~\eqref{it:recursepoint} will ensure it is replaced by one. 
        \item It remains to show the family is layered. By Lemma~\ref{lem:smallimplieslayered} it suffices to check $\mathcal{F}^{\sm}$. Every element $A \in \A^\sm$ will contain the empty set, which is $\A_0$. As to $\A_2^{\sm}$, let $A\in\A_2^{\sm}$ and $B\in\A_1^{\sm}$ such that $B\subset A$. Then since $B$ is either replaced by a subset (steps~\eqref{it:recursepoint} or~\eqref{it:eliminatex}) or is deleted in favor of a subset already in the family (step~\eqref{it:cases}), it must be the case that every member of $\A_2^{\sm}$ is above some member of $\A^{\sm}$. Since the other sets in $\mathcal{F}^{\sm}$ are unchanged, the family is layered. 
    \end{itemize}
    
    So $\A$ is a layered, saturated antichain and $\A^{\sm}$ consists of singletons. It remains to prove that $\bigl|\A^{\lr}\bigr|=1$ and $\bigl|\A^{\sm}\bigr|\geq k-2$.

    Suppose that $\bigl|\A^\lr\bigr| \geq 2$. Consider the element $S$ consisting of all those atoms which are not in $\A^\sm$. By definition of antichain, every member of $\A^{\lr}$ is a subset of $S$. By Lemma~\ref{lem:satant} and the fact that $S$ is above no member of $\A^{\sm}$, it must be either in $\A^{\lr}$ or a subset of a member of that set. But every member of $\A^{\lr}$ must be a subset of $S$ by definition. Thus $S$ is the unique member of $\A^{\lr}$.
    
    Finally, by the definition of layered, there exists a chain from $S$ to $[n]$, each of which are members of $\mathcal{F}$. Specifically, the chain starts with $S$ and contains members of $\A_2^{\lr},\ldots,A_{k-2}^{\lr}$, concluding with $[n]$. This results from adding at least one new element of $[n]\setminus S$ to $S$ each time, $k-2$ times. Thus, there must be at least $k-2$ singletons in $\A^{\sm}$.
\end{proof}

\section{Composition of Saturated Systems}

Once we have a saturated $k$-Sperner system of size less than $2^{k-1}$, there exists a method that gives an upper bound for larger values of $k$.

\begin{lem}[Morrison, Noel, Scott~\cite{MNS}]
    Let $X_1$ and $X_2$ be disjoint sets, and let $\mathcal{F}_1$ and $\mathcal{F}_2$ be saturated $k_1$- and $k_2$-Sperner systems, respectively, and let $H_1 \subset X_1$ and $H_2 \subset X_2$ be homogeneous. Then $$\mathcal{G} := \{A \cup B: A \in \mathcal{F}_1^\sm, B \in \mathcal{F}_2^\sm\} \cup \{S \cup T: S \in \mathcal{F}_1^\lr, T \in \mathcal{F}_2^\lr\}$$
    is a saturated $(k_1 + k_2 - 2)$-Sperner system on $\mathcal{P}(X_1 \cup X_2)$ with homogeneous atom $H_1 \cup H_2$ and cardinality $\bigl|\mathcal{G}\bigr| = \bigl|\mathcal{F}_1^\sm\bigr|\bigl|\mathcal{F}_2^\sm\bigr|
    +\bigl|\mathcal{F}_1^\lr\bigr|\bigl|\mathcal{F}_2^\lr\bigr|$.
    \label{lem:bootstrap}
\end{lem}

Gerbner, et al. introduce a size 30 family which is used to prove a bound for $\sat(k,k)$. With some adjustment, Morrison, Noel, and Scott extend this construction to be a saturated $6$-Sperner set for all $n$.

Repeatedly applying Lemma 15 on this set yields a saturated $k$-Sperner system on an arbitrarily large ground set $X$ such that $$\bigl|\mathcal{F}^\sm\bigr| + \bigl|\mathcal{F}^\lr\bigr| = 15^j + 15^j = 2*15^j$$ for $k = 4j+2$.

For $k$ of the form $4j+2+s$, $0\leq s\leq 3$, we then apply Lemma~\ref{lem:bootstrap} with the saturated 3-Sperner set $\{\emptyset,1,\overline{1} , \overline{\emptyset}\}$ to yield $\sat(k) \leq 2^{s+1}*15^j$.

A very interesting property of this size 30 family is that while the structure of the large and small sets bears a strong symmetric resemblance, they are not complements of each other.

\section{New Upper Bound}
\label{sec:upperbound}

Bootstrapping the above $6$-saturated Sperner example, for $k=7$, gives a $7$-saturated Sperner system of size 60. We have managed to find a $7$-saturated Sperner system of size 56. 

Let $\F =\bigcup_{i=0}^6 \A_i$ be a family composed of the following antichains (see Figure~\ref{fig:constructionlarge} and Figure~\ref{fig:constructionsmall}):
\begin{itemize}
    \item $\A_0 = \{\emptyset\}$ 
    \item $\A_1 = \{2,3,5,6,7\}\cup \{14H\}$ 
    \item $\A_2^{\sm} = \{12, 23, 34, 45, 56, 67, 17\}$ 
    \item $\A_2^{\lr} = \{357H,146H,257H,136H,247H,135H,246H\}$ 
    \item $\A_2 = \A_2^{\sm}\cup A_2^{\lr}$
    \item $\A_3^{\sm} = \{126, 237, 134, 245, 356, 467, 157\}$
    \item $\A_3^{\lr} = \{\overline{S} : S \in \A_3^\sm\}$
    \item $\A_3=\A_3^{\sm}\cup A_3^{\lr}$
    \item $\A_4 = \{\overline{S}: S \in \A_2\}$
    \item $\A_5 = \{\overline{S}: S \in \A_1\}$
    \item $\A_6 = \{\overline{\emptyset}\}$
\end{itemize}

\begin{figure}[t]

    \begin{center}
    
    \begin{tikzpicture}[scale=1,vertex/.style={draw=black, very thick, fill=white, circle, minimum width=2pt, inner sep=2pt, outer sep=1pt},edge/.style={very thick}]
        %\draw[thin,black] (-3.0,-3.4) rectangle (3.0,3.4);
        \coordinate (zero) at (0,-3.3);
        \coordinate (one) at (2.9,3.3);
        \def\hgt{1.45};
        \def\wdtone{1.2};
        \def\wdttwo{1.5};
        \def\wdtthr{2.0};
        \large
        \begin{pgfonlayer}{foreground}
        %\draw[thin,blue] (-6.8,-3.4) rectangle (5.3,3.4);
            \node[vertex] (23567c) at ($(zero)+(-2*\wdttwo,0.2*\hgt)$) {};
            \node[right,xshift=1pt,yshift=1pt] at (23567c) {$\overline{23567}$};
            
            \node[vertex] (1246c) at ($(zero)+(-2.5*\wdttwo,\hgt)$) {};
            \node[left,xshift=-1pt,yshift=1pt] at (1246c) {$\overline{1246}$};
            \node[vertex] (2357c) at ($(zero)+(-2*\wdttwo,\hgt)$) {};
            \node[right,xshift=1pt,yshift=1pt] at (2357c) {$\overline{2357}$};
            \node[vertex] (1346c) at ($(zero)+(-1.0*\wdttwo,\hgt)$) {};
            \node[right,xshift=1pt,yshift=1pt] at (1346c) {$\overline{1346}$};
            \node[vertex] (2457c) at ($(zero)+(0*\wdttwo,\hgt)$) {};
            \node[right, xshift=1pt, yshift=1pt] at (2457c) {$\overline{2457}$};
            \node[vertex] (1356c) at ($(zero)+(0.8*\wdttwo,\hgt)$) {};
            \node[right, xshift=1pt, yshift=1pt] at (1356c) {$\overline{1356}$};
            \node[vertex] (2467c) at ($(zero)+(1.6*\wdttwo,\hgt)$) {};
            \node[right, xshift=1pt, yshift=1pt] at (2467c) {$\overline{2467}$};
            \node[vertex] (1357c) at ($(zero)+(2.5*\wdttwo,\hgt)$) {};
            \node[right, xshift=1pt, yshift=1pt] at (1357c) {$\overline{1357}$};

            \node[vertex] (126c) at ($(zero)+(-2.5*\wdttwo,2*\hgt)$) {};
            \node[left, xshift=-1pt, yshift=1pt] at (126c) {$\overline{126}$};
            \node[vertex] (237c) at ($(zero)+(-2*\wdttwo,2*\hgt)$) {};
            \node[right, xshift=1pt, yshift=1pt] at (237c) {$\overline{237}$};
            \node[vertex] (134c) at ($(zero)+(-1.0*\wdttwo,2*\hgt)$) {};
            \node[right, xshift=1pt, yshift=1pt] at (134c) {$\overline{134}$};
            \node[vertex] (245c) at ($(zero)+(+0*\wdttwo,2*\hgt)$) {};
            \node[right, xshift=1pt, yshift=1pt] at (245c) {$\overline{245}$};
            \node[vertex] (356c) at ($(zero)+(0.8*\wdttwo,2*\hgt)$) {};
            \node[right, xshift=1pt, yshift=1pt] at (356c) {$\overline{356}$};
            \node[vertex] (467c) at ($(zero)+(1.6*\wdttwo,2*\hgt)$) {};
            \node[right, xshift=1pt, yshift=1pt] at (467c) {$\overline{467}$};
            \node[vertex] (157c) at ($(zero)+(2.5*\wdttwo,2*\hgt)$) {};
            \node[right, xshift=1pt, yshift=1pt] at (157c) {$\overline{157}$};

            \node[vertex] (12c) at ($(zero)+(-2.5*\wdttwo,3*\hgt)$) {};
            \node[left, xshift=-1pt, yshift=1pt] at (12c) {$\overline{12}$};
            \node[vertex] (23c) at ($(zero)+(-2*\wdttwo,3*\hgt)$) {};
            \node[right, xshift=1pt, yshift=1pt] at (23c) {$\overline{23}$};
            \node[vertex] (34c) at ($(zero)+(-1*\wdttwo,3*\hgt)$) {};
            \node[right, xshift=1pt, yshift=1pt] at (34c) {$\overline{34}$};
            \node[vertex] (45c) at ($(zero)+(0*\wdttwo,3*\hgt)$) {};
            \node[right, xshift=1pt, yshift=1pt] at (45c) {$\overline{45}$};
            \node[vertex] (56c) at ($(zero)+(0.8*\wdttwo,3*\hgt)$) {};
            \node[right, xshift=1pt, yshift=1pt] at (56c) {$\overline{56}$};
            \node[vertex] (67c) at ($(zero)+(1.6*\wdttwo,3*\hgt)$) {};
            \node[right, xshift=1pt, yshift=1pt] at (67c) {$\overline{67}$};
            \node[vertex] (17c) at ($(zero)+(2.5*\wdttwo,3*\hgt)$) {};
            \node[right, xshift=1pt, yshift=1pt] at (17c) {$\overline{17}$};

            \node[vertex] (2c) at ($(zero)+(-2.5*\wdtone,4*\hgt)$) {};
            \node[left, xshift=-3pt, yshift=0pt] at (2c) {$\overline{2}$};
            \node[vertex] (3c) at ($(zero)+(-1.5*\wdtone,4*\hgt)$) {};
            \node[right, xshift=4pt, yshift=0pt] at (3c) {$\overline{3}$};
            \node[vertex] (5c) at ($(zero)+(0*\wdtone,4*\hgt)$) {};
            \node[left, xshift=-2pt, yshift=0pt] at (5c) {$\overline{5}$};
            \node[vertex] (6c) at ($(zero)+(1.5*\wdtone,4*\hgt)$) {};
            \node[right, xshift=3pt, yshift=0pt] at (6c) {$\overline{6}$};
            \node[vertex] (7c) at ($(zero)+(2.5*\wdtone,4*\hgt)$) {};
            \node[right, xshift=3pt, yshift=0pt] at (7c) {$\overline{7}$};

            \node[vertex] (0c) at ($(zero)+(0*\wdttwo,5*\hgt)$) {};
            \node[above, yshift=1pt] at (0c) {$\overline{\emptyset}$};
        \end{pgfonlayer}
        \begin{pgfonlayer}{background}
            \draw[edge] (0c) -- (3c);
            \draw[edge] (0c) -- (2c);
            \draw[edge] (0c) -- (5c);
            \draw[edge] (0c) -- (6c);
            \draw[edge] (0c) -- (7c);
            
            \draw[edge] (2c) -- (12c);
            \draw[edge] (2c) -- (23c);
            \draw[edge] (3c) -- (23c);
            \draw[edge] (3c) -- (34c);
            \draw[edge] (7c) -- (17c);
            \draw[edge] (5c) -- (56c);
            \draw[edge] (5c) -- (45c);
            \draw[edge] (6c) -- (56c);
            \draw[edge] (6c) -- (67c);
            \draw[edge] (7c) -- (67c);
            
            \draw[edge] (12c) -- (126c);
            \draw[edge] (23c) -- (237c);
            \draw[edge] (34c) -- (134c);
            \draw[edge] (45c) -- (245c);
            \draw[edge] (56c) -- (356c);
            \draw[edge] (67c) -- (467c);
            \draw[edge] (17c) -- (157c);
            
            \draw[edge] (126c) -- (1246c);
            \draw[edge] (237c) -- (2357c);
            \draw[edge] (134c) -- (1346c);
            \draw[edge] (245c) -- (2457c);
            \draw[edge] (356c) -- (1356c);
            \draw[edge] (467c) -- (2467c);
            \draw[edge] (157c) -- (1357c);
            
            \draw[edge] (2357c) -- (23567c);
        \end{pgfonlayer}
    \end{tikzpicture}
\end{center}

    \caption{Saturated 7-Sperner system, $\mathcal{F}^{\lr}$.  Note: $\overline{1246}=357H$, $\overline{2357}=146H$, $\overline{1346}=257H$, $\overline{2457}=136H$, $\overline{1356}=247H$, and $\overline{2467}=135H$, $\overline{1357}=246H$.}
    \label{fig:constructionlarge}
    \end{figure}
    
    \begin{figure}[t]
    \begin{tikzpicture}[scale=1,vertex/.style={draw=black, very thick, fill=white, circle, minimum width=2pt, inner sep=2pt, outer sep=1pt},edge/.style={very thick}]
        %\draw[thin,black] (-3.0,-3.4) rectangle (3.0,3.4);
        \coordinate (zero) at (0,-3.3);
        \coordinate (one) at (2.9,3.3);
        \def\hgt{1.45};
        \def\wdtone{1.2};
        \def\wdttwo{1.5};
        \def\wdtthr{2.0};
        \large
        \begin{pgfonlayer}{foreground}
        %\draw[thin,blue] (-6.8,-3.4) rectangle (5.3,3.4);
            \node[vertex] (0) at ($(zero)+(0,0.2*\hgt)$) {};
            \node[below, yshift=-1pt] at (0) {$\emptyset$};
            \node[vertex] (2) at ($(zero)+(-2.5*\wdtone,\hgt)$) {};
            \node[left, xshift=-3pt, yshift=1pt] at (2) {$2$};
            \node[vertex] (3) at ($(zero)+(-1.5*\wdtone,\hgt)$) {};
            \node[right, xshift=4pt, yshift=1pt] at (3) {$3$};
            \node[vertex] (5) at ($(zero)+(0*\wdtone,\hgt)$) {};
            \node[left, xshift=-2pt, yshift=1pt] at (5) {$5$};
            \node[vertex] (6) at ($(zero)+(1.5*\wdtone,\hgt)$) {};
            \node[right, xshift=3pt, yshift=1pt] at (6) {$6$};
            \node[vertex] (7) at ($(zero)+(2.5*\wdtone,\hgt)$) {};
            \node[right, xshift=3pt, yshift=1pt] at (7) {$7$};

            \node[vertex] (12) at ($(zero)+(-2.5*\wdttwo,2*\hgt)$) {};
            \node[left, xshift=-1pt, yshift=0pt] at (12) {$12$};
            \node[vertex] (23) at ($(zero)+(-2*\wdttwo,2*\hgt)$) {};
            \node[right, xshift=1pt, yshift=0pt] at (23) {$23$};
            \node[vertex] (34) at ($(zero)+(-1.0*\wdttwo,2*\hgt)$) {};
            \node[right, xshift=1pt, yshift=0pt] at (34) {$34$};
            \node[vertex] (45) at ($(zero)+(+0*\wdttwo,2*\hgt)$) {};
            \node[right, xshift=1pt, yshift=0pt] at (45) {$45$};
            \node[vertex] (56) at ($(zero)+(0.8*\wdttwo,2*\hgt)$) {};
            \node[right, xshift=1pt, yshift=0pt] at (56) {$56$};
            \node[vertex] (67) at ($(zero)+(1.6*\wdttwo,2*\hgt)$) {};
            \node[right, xshift=1pt, yshift=0pt] at (67) {$67$};
            \node[vertex] (17) at ($(zero)+(2.5*\wdttwo,2*\hgt)$) {};
            \node[right, xshift=1pt, yshift=0pt] at (17) {$17$};

            \node[vertex] (126) at ($(zero)+(-2.5*\wdttwo,3*\hgt)$) {};
            \node[left, xshift=-1pt, yshift=0pt] at (126) {$126$};
            \node[vertex] (237) at ($(zero)+(-2*\wdttwo,3*\hgt)$) {};
            \node[right, xshift=1pt, yshift=0pt] at (237) {$237$};
            \node[vertex] (134) at ($(zero)+(-1*\wdttwo,3*\hgt)$) {};
            \node[right, xshift=1pt, yshift=0pt] at (134) {$134$};
            \node[vertex] (245) at ($(zero)+(0*\wdttwo,3*\hgt)$) {};
            \node[right, xshift=1pt, yshift=0pt] at (245) {$245$};
            \node[vertex] (356) at ($(zero)+(0.8*\wdttwo,3*\hgt)$) {};
            \node[right, xshift=1pt, yshift=0pt] at (356) {$356$};
            \node[vertex] (467) at ($(zero)+(1.6*\wdttwo,3*\hgt)$) {};
            \node[right, xshift=1pt, yshift=0pt] at (467) {$467$};
            \node[vertex] (157) at ($(zero)+(2.5*\wdttwo,3*\hgt)$) {};
            \node[right, xshift=1pt, yshift=0pt] at (157) {$157$};

            \node[vertex] (1246) at ($(zero)+(-2.5*\wdttwo,4*\hgt)$) {};
            \node[left, xshift=-1pt, yshift=0pt] at (1246) {$1246$};
            \node[vertex] (2357) at ($(zero)+(-2*\wdttwo,4*\hgt)$) {};
            \node[right, xshift=1pt, yshift=0pt] at (2357) {$2357$};
            \node[vertex] (1346) at ($(zero)+(-1*\wdttwo,4*\hgt)$) {};
            \node[right, xshift=1pt, yshift=0pt] at (1346) {$1346$};
            \node[vertex] (2457) at ($(zero)+(0*\wdttwo,4*\hgt)$) {};
            \node[right, xshift=1pt, yshift=0pt] at (2457) {$2457$};
            \node[vertex] (1356) at ($(zero)+(0.8*\wdttwo,4*\hgt)$) {};
            \node[right, xshift=1pt, yshift=0pt] at (1356) {$1356$};
            \node[vertex] (2467) at ($(zero)+(1.6*\wdttwo,4*\hgt)$) {};
            \node[right, xshift=1pt, yshift=0pt] at (2467) {$2467$};
            \node[vertex] (1357) at ($(zero)+(2.5*\wdttwo,4*\hgt)$) {};
            \node[right, xshift=1pt, yshift=0pt] at (1357) {$1357$};

            \node[vertex] (23567) at ($(zero)+(-2*\wdttwo,5*\hgt)$) {};
            \node[right, xshift=1pt, yshift=0pt] at (23567) {$23567$};
        \end{pgfonlayer}
        \begin{pgfonlayer}{background}
            \draw[edge] (0) -- (3);
            \draw[edge] (0) -- (2);
            \draw[edge] (0) -- (5);
            \draw[edge] (0) -- (6);
            \draw[edge] (0) -- (7);
            \draw[edge] (2) -- (12);
            \draw[edge] (2) -- (23);
            \draw[edge] (3) -- (23);
            \draw[edge] (3) -- (34);
            \draw[edge] (7) -- (17);
            \draw[edge] (5) -- (56);
            \draw[edge] (5) -- (45);
            \draw[edge] (6) -- (56);
            \draw[edge] (6) -- (67);
            \draw[edge] (7) -- (67);
            
            \draw[edge] (12) -- (126);
            \draw[edge] (23) -- (237);
            \draw[edge] (34) -- (134);
            \draw[edge] (45) -- (245);
            \draw[edge] (56) -- (356);
            \draw[edge] (67) -- (467);
            \draw[edge] (17) -- (157);
            
            \draw[edge] (126) -- (1246);
            \draw[edge] (237) -- (2357);
            \draw[edge] (134) -- (1346);
            \draw[edge] (245) -- (2457);
            \draw[edge] (356) -- (1356);
            \draw[edge] (467) -- (2467);
            \draw[edge] (157) -- (1357);
            
            \draw[edge] (2357) -- (23567);
        \end{pgfonlayer}
    \end{tikzpicture}
    \caption{Saturated 7-Sperner system, $\mathcal{F}^{\sm}$.}
    \label{fig:constructionsmall}
\end{figure}

In order to verify that this construction is a $7$-saturated Sperner system, we first need to verify the layering condition for which we refer the reader to Figure~\ref{fig:constructionlarge} and Figure~\ref{fig:constructionsmall}.

Finally, we must verify that each $\A_i$ is a saturated antichain. For $\A_0$ and $\A_6$ this is trivial. For $\A_1$ it is easy to see that every set is either a superset of a singleton in $\A_1^{\sm}=\{2,3,5,6,7\}$ or a subset of the member of $\A_1^{\lr}=\{14H\}$. By symmetry, the saturation of $\A_5$ follows.

For $\A_2$, any set of size at least $3$ must either contain some pair in $\A_2^{\sm}$ (observe that the pairs in $\A_2^{\sm}$ are consecutive atoms modulo 7) or is a set of size exactly $3$ with no pair of consecutive atoms modulo 7, which is what forms $\A_2^{\lr}$. Each pair or singleton is a subset of a member of $\A_2^{\lr}$ except those pairs of consecutive atoms, which is exactly $\A_2^{\sm}$. Thus $\A_2$ is a saturated antichain. By symmetry, $\A_4$ is a saturated antichain also.

Finally, the middle antichain $\A_3$ is the union of the Fano plane and the family consisting of the complements of those sets. Because the chromatic number of the Fano plane is 3, every subset of the seven vertices either contains a hyperedge of the Fano plane or is contained in the complement of a hyperedge. Thus, $\A_3$ is a saturated antichain. 

It is worth mentioning that there exist smaller saturated antichains with small and large elements of the same cardinality as those in $\A_2$; however, these cannot be layered underneath the Fano plane as $\A_2$ is. It should also be noted that antichains that are layered underneath the Fano plane of size 14 are plentiful, and may have a different distribution of elements between $\A_2^\sm$ and $\A_2^\lr$.

\begin{figure}[h]
\begin{center}
    \begin{tikzpicture}[scale=1,vertex/.style={draw=black, very thick, fill=white, circle, minimum width=2pt, inner sep=2pt, outer sep=1pt},edge/.style={very thick}]
         \coordinate (zero) at (1,2);
        \coordinate (one) at (7,2);
        \def\xlen{3}
        \def\wdt{\xlen/2};
        \def\hgt{\wdt/1.732};
        \def\wdttwo{\wdt/2}
        \def\hgttwo{\wdt*0.866}
        \def\rad{\hgt*2}
        
        \def\sc{0.5}
        \begin{pgfonlayer}{foreground}
            \node[vertex] (6) at ($(zero)$) {};
            \node[right,xshift=-2pt,yshift = -9] at (6) {$6$};
            \node[vertex] (2) at ($(zero)+(0,\rad)$) {};
            \node[right,xshift=1pt] at (2) {$2$};
            \node[vertex] (4) at ($(zero)+(-\wdt,-\hgt)$) {};
            \node[above,xshift=-5pt] at (4) {$4$};
            \node[vertex] (3) at ($(zero)+(\wdt,-\hgt)$)  {};
            \node[above,xshift=5pt] at (3) {$3$};
            \node[vertex] (5) at ($(zero)+(-\wdt+\wdttwo,\hgttwo-\hgt)$)  {};
            \node[left,yshift=-3pt,xshift=-2pt] at (5) {$5$};
            \node[vertex] (7) at ($(zero)+(\wdt-\wdttwo,\hgttwo-\hgt)$) {};
            \node[right,yshift=-3pt,xshift=2pt] at (7) {$7$};
            \node[vertex] (1) at ($(zero)+(0,-\hgt)$) {};
            \node[below,xshift=-4pt] at (1) {$1$};
            
            \node[vertex] (6a) at ($(one)$) {};
            \node[above,yshift=1pt,yshift =2pt] at (6a) {$6$};
            \node[vertex] (2a) at ($(one)+(0,\rad)$) {};
            \node[above,yshift=1pt] at (2a) {$2$};
            \node[vertex] (4a) at ($(one)+(-\wdt,-\hgt)$) {};
            \node[below,xshift=-5pt] at (4a) {$4$};
            \node[vertex] (3a) at ($(one)+(\wdt,-\hgt)$)  {};
            \node[below,xshift=5pt] at (3a) {$3$};
            \node[vertex] (5a) at ($(one)+(-\wdt+\wdttwo,\hgttwo-\hgt)$)  {};
            \node[left,xshift=-2pt] at (5a) {$5$};
            \node[vertex] (7a) at ($(one)+(\wdt-\wdttwo,\hgttwo-\hgt)$) {};
            \node[right,xshift=2pt] at (7a) {$7$};
            \node[vertex] (1a) at ($(one)+(0,-\hgt)$) {};
            \node[below,yshift=-1pt] at (1a) {$1$};;
        \end{pgfonlayer}
        \begin{pgfonlayer}{background}
            \draw[edge] ($(1)+(0,-0.5)$) to ($(2)+(0,0.5)$);
            \draw[edge] ($(3)+(\sc,0)$) to ($(4)+(-\sc,0)$);
            \draw[edge] (1) to [out = 0, in = 260] ($(7)+(0.08,0.4)$);
            \draw[edge] (1) to [out = 180, in = 280] ($(5)+(-.08,.4)$);
            \draw[edge] ($(3)+(\sc,-\sc+0.25)$) to ($(5)+(-\sc,\sc-0.25)$);
            \draw[edge] ($(3)+(\sc,-\sc-0.25)$) to ($(2)+(-\sc,\sc+0.25)$);
            \draw[edge] ($(2)+(\sc,\sc+0.25)$) to ($(4)+(-\sc,-\sc-0.25)$);
            \draw[edge] ($(4)+(-\sc,-\sc+0.25)$) to ($(7)+(\sc,\sc-0.25)$);
            
            \draw[edge] (1a) to [bend left] (2a);
            \draw[edge] (1a) -- (7a);
            \draw[edge] (6a) -- (7a);
            \draw[edge] (5a) -- (6a);
            \draw[edge] (4a) -- (5a);
            \draw[edge] (3a) to [bend right=40] (2a);
            \draw[edge] (3a) to [bend left=45] (4a);
        \end{pgfonlayer}
    \end{tikzpicture}
\end{center}
\caption{The Fano plane ($\A_3^\sm$) and $\A_2^\sm$}
\end{figure}

As this is a layered sequence of pairwise disjoint saturated antichains, Lemma~\ref{lem:antichainsareSperner} gives that $\mathcal{F}$ is a saturated $7$-Sperner system of size 56. Using Lemma~\ref{lem:bootstrap}, we can construct a saturated $k$-Sperner system of cardinality $2*28^j*2^s$ for $k = 5j + 2 +s$, $0\leq s\leq 4$ on a sufficiently large ground set.

Theorem~\ref{thm:UB} is obtained by observing $\sat(k)\leq 2^{s+1}\cdot 28^j$ and solving $2^{s+1}\cdot 28^j\leq 2^{(1-\varepsilon)k}$ for $\varepsilon$. A value of $\varepsilon = \left(1-\frac{\log_2(28)}{5} \right) \approx 0.038529$ is sufficient to make this true for all $k\geq 7$.

\section{New lower bound}

As for the lower bound, we will assume, via Proposition~\ref{prop:sizes}, that $\F$ consists of saturated antichains $\A_i$, $i=0,\ldots,k-1$ where $\A_0=\{\emptyset\}$, $\A_{k-1}=\{[n]\}$ and for all $i\in\{1,\ldots,k-2\}$, $\A_i$ has the property that every $S\in\A_i^{\sm}$ has cardinality at least $i$ and every $L\in\A_i^{\lr}$ has cardinality at most $n-(k-i-1)$. 

First, Lemma~\ref{lem:antichainLB} establishes a lower bound for each of the $\A_i$s. Recall by Lemma~\ref{lem:A1} that $\bigl|\A_0\bigr|=\bigl|\A_{k-1}\bigr|=1$, $\bigl|\A_1^{\sm}\bigr|,\bigl|\A_{k-2}^{\lr}\bigr|\geq k-2$, and  $\bigl|\A_1^{\lr}\bigr|=\bigl|\A_{k-2}^{\sm}\bigr|=1$.

\begin{lem}
    Let $k \geq 7$ and $i$ be integers, where $2 \leq i \leq \lfloor(k-1)/2\rfloor$, and let $\A_i^\sm \cup \A_i^\lr$ be a saturated antichain such that $|S| \geq i$ for all $S \in \A_i^\sm$ and $|L| \leq n- (k-i-1)$ for all $L \in \A_i^\lr$.

    Then $|\A_i^\sm| + |\A_i^\lr| > \exp\left\{2\ln 2\frac{i(k-i-1)}{k-1}\right\}$.
    \label{lem:antichainLB}
\end{lem}
\begin{proof}
    Choose a set $R$ at random by choosing each member of $\{1,\ldots,n\}$ independently with probability $1/2-\varepsilon$, where $\varepsilon = \frac{\ln2}{2}\left(1-\frac{2i}{k-1}\right)$.

    Let $X$ be the number of $S \in \A_i^\sm$ such that $S \subseteq R$ plus the number of $L \in \A_i^\lr$ such that $R \subseteq L$.

    The expected value of $X$ will be
    \begin{align*}
        \mathbb{E}[X] &= \sum_{S \in \A_i^\sm} \left(\frac{1}{2}-\varepsilon\right)^{|S|} + \sum_{L \in \A_i^\lr} \left(\frac{1}{2}+\varepsilon\right)^{n-|L|}\\
        \mathbb{E}[X] &\leq \bigl|\A_i^\sm\bigr|\left(\frac{1}{2}-\varepsilon\right)^i + \bigl|\A_i^\lr\bigr|\left(\frac{1}{2}+\varepsilon\right)^{k-i-1}\\
        &< \bigl|\A_i^\sm\bigr|\exp\left\{-i\ln 2 - 2i\varepsilon\right\} + \bigl|\A_i^\lr\bigr|\exp\left\{-(k-i-1)\ln2 + 2(k-i-1)\varepsilon\right\}\\
        &\leq \left(\bigl|\A_i^\sm\bigr| + \bigl|\A_i^\lr\bigr|\right)\exp\left\{-2\ln2 \frac{i(k-i-1)}{k-1}\right\}.\\
    \end{align*}
    Since the antichain is saturated, $\mathbb{E}[X] \geq 1$ and the result follows.
\end{proof}

The more specific version of Theorem~\ref{thm:LB} is as follows:
\begin{thm}
    Let $k \geq 7$. If $\F$ is a saturated $k$-Sperner system, then $|\F| \geq 2^{k/2 + (1/2)\log_2k -1.66}$ and $|\F| \geq 2^{k/2 + (1/2)\log_2k}$ if $k\geq 497$. 
\end{thm}

\begin{proof}
    We know that $\bigl|\A_0\bigr| = \bigl|\A_{k-1}\bigr| = 1$, $\bigl|\A_1\bigr|, \bigl|\A_{k-2}\bigl|\geq k-1$. The remaining layers have sizes bounded by Lemma~\ref{lem:antichainLB}, and the bounds for $|\A_i|$ will be the same as for $|\A_{k-1-i}|$. Using the notation ${\bf 1}_{k\;{\rm odd}}$ to mean 1 if $k$ is odd and 0 if $k$ is even, then
    \begin{align*}
        |\F| &\geq 2 + 2(k-1) + 2 \sum_{i=2}^{\lceil(k-1)/2\rceil -1}\exp\left\{2\ln2 \frac{i(k-i-1)}{k-1}\right\} + \exp\left\{\ln2 \frac{k-1}{2}\right\}{\bf 1}_{k\;{\rm odd}} .
    \end{align*}

    We will lower-bound by zero the terms outside of the main summation, as this does not effect the asymptotics. Since the function $\exp\left\{2\ln 2\frac{x(k-x-1)}{k-1}\right\}$ is increasing for all $x\in \left(2,(k-1)/2\right)$, it may be estimated by an integral. As such,
    \begin{align*}
        |\F| &\geq 2\int_{x=1}^{(k-3)/2} \exp\left\{2\ln 2\; \frac{x(k-1-x)}{k-1}\right\}\, dx .
    \end{align*}
%   Note that the upper limit should be $\lceil(k-1)/2\rceil-1$ but this is lower bounded by $(k-3)/2$

    A substitution allows us to put this in terms of the error function, that is, $\erf(x)={\displaystyle\int_0^x} \frac{2}{\sqrt{\pi}} e^{-y^2}\, dy$. The substitution is $x=\frac{k-1}{2}-y\sqrt{\frac{k-1}{2\ln 2}}$

    \begin{align*}
        |\F| &\geq 2\int_{y=(k-3)\sqrt{\ln 2/(2(k-1))}}^{\sqrt{2\ln 2/(k-1)}} \exp\left\{\frac{2\ln 2}{k-1}\; \left(\left(\frac{k-1}{2}\right)^2-y^2\frac{k-1}{2\ln 2}\right)\right\} 
        \left(-\sqrt{\frac{k-1}{2\ln 2}}\right)\, dy \\
        &\geq 2\exp\left\{\frac{(k-1)\ln 2}{2}\right\} \sqrt{\frac{k-1}{2\ln 2}}\int_{y=\sqrt{2\ln 2/(k-1)}}^{(k-3)\sqrt{\ln 2/(2(k-1))}} e^{-y^2}\, dy \\
        &\geq 2^{k/2} \sqrt{k} \sqrt{\frac{\pi (k-1)}{4k\ln 2}} \int_{y=\sqrt{2\ln 2/(k-1)}}^{(k-3)\sqrt{\ln 2/(2(k-1))}} \frac{2}{\sqrt{\pi}} e^{-y^2} dy \\
        &\geq 2^{k/2} \sqrt{k} \sqrt{\frac{\pi (k-1)}{4k\ln 2}} \left[{\textstyle \left(\erf\left\{\sqrt{\frac{(k-3)^2\ln 2}{2(k-1)}}\right\} - \erf\left\{\sqrt{\frac{2\ln 2}{k-1}}\right\} \right)}\right] \\
        &\geq 2^{k/2+(1/2)\log_2 k-1.66}\\
    \end{align*}
    Moreover, if $k \geq 497$, then a computer algebra system gives $|\F| \geq 2^{k/2 + (1/2)\log_2k}$.
\end{proof}

\section{Further Questions}
There are a number of open questions related to this problem. 
\begin{itemize}
    \item Is $\sat(7) = 56$, or is there a smaller construction?
    \item Can the bounds of $\sqrt{k}2^{k/2}<\sat(k)<2^{0.961471 k}$ be improved for large $k$?
    \item Are all of the optimal Sperner systems \emph{flat}, meaning that in each saturated antichain, the small elements have the same size as each other, and the large elements have the same size as each other?
\end{itemize}

\section{Acknowledgements}
We are indebted to Deepak Bal, Jonathan Cutler, and Andrew Michel, who found an error in a previous version of this manuscript which falsely claimed a lower value for $\sat(7)$. 

Martin's research was partially supported by a grant from the Simons Foundation \#709641 and this research was partially done while Martin was on an MTA Distinguished Guest Scientist Fellowship 2023 at the HUN-REN Alfr\'ed R\'enyi Institute of Mathematics. Veldt's research was supported by a grant from the National Science Foundation DMS-1839918 (RTG). Both authors are indebted to the Alfr\'ed R\'enyi Institute of Mathematics for hosting Veldt for a collaboration visit. They would also like to thank Bal\'azs Patk\'os for helpful comments on the manuscript.

\bibliographystyle{alpha}
\bibliography{bibli}

\end{document}